\documentclass[12pt]{article}

\usepackage{amsfonts, amssymb, amsmath, amsthm, epsf}

\makeatletter

\@addtoreset{figure}{section}

\theoremstyle{plain}
\newtheorem{theorem}{Theorem}
\@addtoreset{theorem}{section}

\newtheorem{lemma}{Lemma}
\@addtoreset{lemma}{section}

\@addtoreset{corollary}{section}

\@addtoreset{condition}{section}

\theoremstyle{definition}

\newtheorem{definition}{Definition}
\@addtoreset{definition}{section}

\@addtoreset{example}{section}

\newtheorem{remark}{Remark}
\@addtoreset{remark}{section}

\@addtoreset{equation}{section}

\makeatother

\renewcommand{\Im}{{\rm Im\,}}

\newcommand{\ind}{{\rm ind\,}}
\newcommand{\supp}{{\rm supp\,}}
\newcommand{\Dom}{{\rm D\,}}
\renewcommand{\ker}{{\rm ker\,}}
\renewcommand{\dim}{{\rm dim\,}}
\newcommand{\codim}{{\rm codim\,}}

\newcommand{\const}{{\rm const}}
\newcommand{\dist}{{\rm dist}}

\title{On the index instability\\ for some
nonlocal elliptic problems}
\author{Pavel Gurevich\thanks{Supported by the
Russian Foundation for Basic Research (project No.~04-01-00256)
and by the Russian President's grant (project No.~MK-980.2005.1).}
\thanks{Email address: gurevichp@yandex.ru}}

\date{}

\begin{document}

\maketitle

\begin{abstract}
The Fredholm index of unbounded operators defined on generalized
solutions of nonlocal elliptic problems in plane bounded domains
is investigated. It is known that nonlocal terms with smooth
coefficients having zero of a certain order at the conjugation
points do not affect the index of the unbounded operator. In this
paper, we construct examples showing that the index may change
under nonlocal perturbations with coefficients not vanishing at
the points of conjugation of boundary-value conditions.
\end{abstract}

\section{Introduction}

The first one who began to study nonlocal problems in
multidimensional case was Carleman. In~\cite{Carleman}, the
problem of finding a function harmonic on a two-dimensional
bounded domain and subjected to a nonlocal condition connecting
the values of this function at different points of the boundary is
considered. Bitsadze and Smarskii~\cite{BitzSam} suggested another
setting of a nonlocal problem arising in plasma theory: to find a
function harmonic on a bounded domain and satisfying nonlocal
conditions on shifts of the boundary that can take points of the
boundary inside the domain. Different generalizations of the above
nonlocal problems were investigated by many authors
(see~\cite{SkBook} and references therein).

It turns out that the most difficult situation occurs if the
support of nonlocal terms intersects the boundary. In this case,
solutions of nonlocal problems can have power-law singularities
near some points even if the boundary and the right-hand sides are
infinitely smooth~\cite{SkMs86}. For this reason, such problems
are naturally studied in weighted spaces introduced by
Kondrat'ev~\cite{KondrTMMO67} for boundary-value problems in
nonsmooth domains. In particular, the Fredholm solvability of
general nonlocal elliptic problems in weighted spaces is
investigated by Skubachevskii~\cite{SkMs86,SkDu90,SkDu91} and his
pupils.

In \cite{GurMatZam05}, an unbounded operator $\mathbf P: L_2(G)\to
L_2(G)$ corresponding to an elliptic equation of order $2m$ in a
plane bounded domain $G\subset\mathbb R^2$ with nonlocal
boundary-value conditions is studied; the operator $\mathbf P$ is
defined on generalized solutions of the nonlocal problem, i.e.,
the domain $\Dom(\mathbf P)$ consists of the functions $u\in
W_2^m(G)$ that satisfy nonlocal conditions in the sense of traces
and for which $\mathbf P u\in L_2(G)$. In particular, it is proved
that the operator $\mathbf P$ has the Fredholm property.

In~\cite{GurTrMIRAN}, we investigate how lower-order terms in
elliptic equations and nonlocal perturbations in boundary-value
conditions affect the index of the unbounded operator $\mathbf P$.
It is proved that the index of $\mathbf P$ does not change if we
add nonlocal terms with smooth coefficients having zero of a
certain order at the points of conjugation of boundary-value
conditions. It is also proved that lower-order terms in elliptic
equation have no influence on the index.

In this paper, we construct examples showing that the index may
change under nonlocal perturbations with coefficients not
vanishing at the points of conjugation of boundary-value
conditions (even if the coefficients in the perturbations are
arbitrarily small). The case where the support of nonlocal terms
contains the conjugation points and the case where it does not are
both considered.

The reason of the index instability is that nonlocal terms change
the domain of the corresponding unbounded operator; it turns out
that if the coefficients at nonlocal terms do not vanish at the
conjugation points, then these terms cannot be reduced to
(relatively) small or compact perturbations.

\section{Preliminaries}

Let $G$ be a bounded domain in $\mathbb R^2$, and let
$$\partial G\setminus\{g_1,g_2\}=\Gamma_1\cup\Gamma_2,$$ where
$\Gamma_i$ are open (in the topology of $\partial G$) infinitely
smooth curves,
$\overline{\Gamma_1}\cap\overline{\Gamma_2}=\{g_1,g_2\}$, and
$g_1$ and $g_2$ are the endpoints of the curves
$\overline{\Gamma_1}$ and $\overline{\Gamma_2}$. We assume that
the domain $G$ is a plane angle of nonzero opening in a
neighborhood of the points $g_j$.

For integer $k\ge0$, we denote by $W_2^k(G)$ the Sobolev space
with the norm
$$
\|u\|_{W_2^k(G)}=\left(\sum\limits_{|\alpha|\le k}\int\limits_G
|D^\alpha u|^2\,dy\right)^{1/2}
$$
(set $W_2^0(G)=L_2(G)$ for $k=0$). For integer $k\ge1$, we
introduce the space $W_2^{k-1/2}(\Gamma)$ of traces on a smooth
curve $\Gamma\subset\overline{G}$ with the norm
\begin{equation}\label{eqTraceNormW}
\|\psi\|_{W_2^{k-1/2}(\Gamma)}=\inf\|u\|_{W_2^k(G)},
\end{equation}
where the infimum is taken over all $u\in W_2^k(G)$ such that
$u|_\Gamma=\psi$.

Let $C_0^\infty(\overline{G}\setminus\{g_1,g_2\})$ denote the set
of functions infinitely differentiable on $\overline{G}$ and
vanishing near the points $g_1$ and $g_2$.

We introduce the Kondrat'ev space $H_a^k(G)$ as the completion of
the set $C_0^\infty(\overline{G}\setminus\{g_1,g_2\})$ with
respect to the norm
$$
\|u\|_{H_a^k(G)}=\left(\sum\limits_{|\alpha|\le s}\int\limits_G
\rho^{2(a+|\alpha|-k)}|D^\alpha u|^2\, dx\right)^{1/2},
$$
where $k\ge0$, $a\in\mathbb R$, and
$\rho(y)=\dist(y,\{g_1,g_2\})$.

Denote by $H_a^{k-1/2}(G)$ ($k\ge1$ is an integer) the space of
traces on a smooth curve $\Gamma\subset\overline G$ with the norm
$$
\|\psi\|_{H_a^{k-1/2}(\Gamma)}=\inf\|u\|_{H_a^k(G)},
$$
where the infimum is taken over all $u\in H_a^k(G)$ such that
$u|_\Gamma=\psi$.

It is clear that $H_0^k(G)\subset W_2^k(G)$ for integer
nonnegative $k$ and $H_0^{k-1/2}(\Gamma)\subset
W_2^{k-1/2}(\Gamma)$ for integer $k\ge1$. In the sequel, we also
need the following result about the embedding of Sobolev spaces
into weighted spaces.

\begin{lemma}[see Lemma 4.11 in \cite{KondrTMMO67}]\label{lPsi=P1+Phi}
The operators of embedding $W_2^{3/2}(\Gamma_j)\subset
H_{a+1}^{3/2}(\Gamma_j)$, $j=1,2$, are bounded for all $a>0$.
\end{lemma}

To conclude this section, we formulate a theorem about small
perturbations of Fredholm operators in Hilbert spaces.

Let $H_1$ and $H_2$ be Hilbert spaces, and let $ P:\Dom(P)\subset
H_1\to H_2 $ be a linear (generally speaking, unbounded) operator.

\begin{definition}\label{defFredholm}
The operator $P$ is said {\it to have the Fredholm property} if it
is closed, its image is closed, and the dimension of its kernel
$\ker P$ and the codimension of its image $\mathcal R(P)$ are
finite. The number $\ind P=\dim\ker P-\codim\mathcal R(P)$ is
called the {\it index} of the Fredholm operator $P$.
\end{definition}

Let $ A: H_1\to H_2 $ be a linear operator.

\begin{theorem}[see Sec.~16 in~\cite{Kr}]\label{thAppBTheorSec16Kr} Let the
operator $P$ have the Fredholm property, $A$ be bounded, and
$\Dom(A)=H_1$. Then the operator $P+A$ has the Fredholm property,
$\ind(P+A)=\ind P$, $\dim\ker(P+A)\le\dim\ker P$, and
$\codim\mathcal R(P+A)\le \codim\mathcal R(P)$, provided that
$\|A\|$ is sufficiently small.
\end{theorem}

\section{The index instability in the case where
the support of nonlocal terms contains the conjugation points}

{\bf 1.} Let $G$, $\Gamma_i$, and $g_j$ be the same as above. Let
a number ${\varepsilon}>0$ be so small that the
$\varepsilon$-neighborhoods  $\mathcal O_{\varepsilon}(g_j)$ of
the points $g_j$ do not intersect with each other. We assume that
the set $G\cap \mathcal O_{\varepsilon}(g_j)$ coincides with the
plane angle of opening $2\omega_0$, where $0<\omega_0<\pi$.
Consider the following nonlocal problem in the domain $G$:
\begin{equation}\label{eqEx1_61}
 \Delta u=f(y),\quad y\in G,
\end{equation}
\begin{equation}\label{eqEx1_62}
u|_{\Gamma_1}-(1+t) u(\Omega_1(y))|_{\Gamma_1}=0,\qquad
u|_{\Gamma_2}-(1-t) u(\Omega_2(y))|_{\Gamma_2}=0.
\end{equation}
Here $t\in\mathbb C$ is a parameter and $\Omega_i$ is a
$C^\infty$-diffeomorphism defined on a neighborhood of $\Gamma_i$.
We assume that $\Omega_i(\Gamma_i)\subset G$, $\Omega_i(g_1)=g_1$,
$\Omega_i(g_2)=g_2$, and the transformation $\Omega_i$ is the
rotation by the angle $\omega_0$ inside the domain $G$ near the
points $g_1$ and $g_2$ (see Fig.~\ref{figEx1}).
\begin{figure}[ht]
{ \hfill\epsfxsize10cm\epsfbox{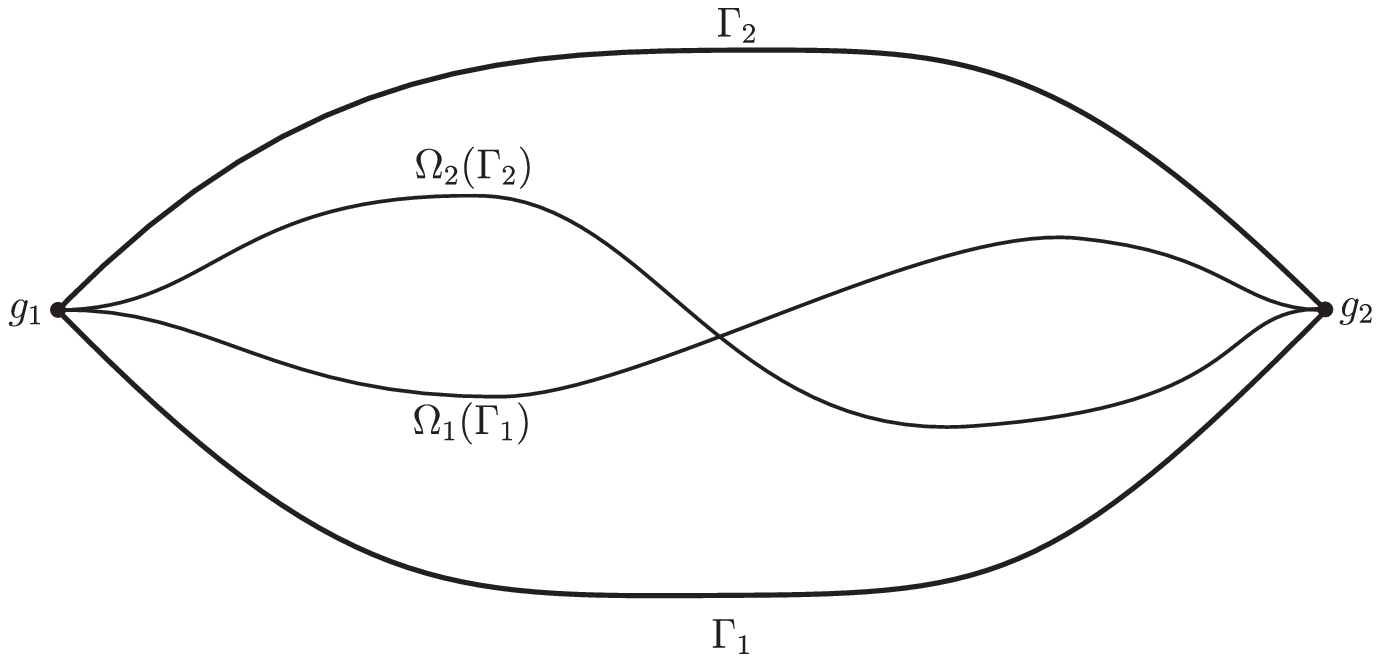}\hfill\ }
\caption{Problem~\eqref{eqEx1_61}, \eqref{eqEx1_62}}
\label{figEx1}
\end{figure}

We say that a function $u\in W_2^1(G)$ is a {\it generalized
solution} of problem~\eqref{eqEx1_61}, \eqref{eqEx1_62} with
right-hand side $f\in L_2(G)$ if $u$ satisfies
Eq.~\eqref{eqEx1_61} in the sense of distributions and nonlocal
conditions~\eqref{eqEx1_62} in the sense of traces.

\smallskip

Consider the unbounded operator $\mathbf P_t: \Dom(\mathbf
P_t)\subset L_2(G)\to L_2(G)$ given by
$$
 \mathbf P_t u=\Delta u, \qquad u\in \Dom(\mathbf P_t),
$$
where
$$
\Dom(\mathbf P_t)=\{u\in
 W_2^1(G):\ \Delta u\in L_2(G)\ \text{and}\ u\ \text{satisfies~\eqref{eqEx1_62}}
 \}.
$$
The operator $\mathbf P_t$ has the Fredholm property for all
$t\in\mathbb C$ by Theorem~2.1 in~\cite{GurMatZam05}.

Let us prove the following result.

\begin{theorem}\label{thEx1_IndUnstab}
There is a number $t_0>0$ such that $\ind \mathbf P_0>\ind\mathbf
P_t=\const$ for $0<|t|\le t_0$.
\end{theorem}

{\bf 2.} It is known (see, e.g.,~\cite{SkMs86}) that the behavior
of solutions for problem~\eqref{eqEx1_61}, \eqref{eqEx1_62} near
the points $g_1$ and $g_2$ depends on the location of eigenvalues
of model problems (with a parameter) corresponding to these
points. To write the model problem corresponding to the point
$g_1$, we assume that $g_1$ is the origin and the axis $Ox_1$
coincides with the bisectrix of the angle formed by the boundary
of the domain $G$ near the point $g_1$. Consider
problem~\eqref{eqEx1_61}, \eqref{eqEx1_62} for $y\in\mathcal
O_\varepsilon(g_1)$, formally setting $f=0$:
\begin{equation*}\label{eqEx1_61Eps}
 \Delta u=0,\quad y\in\mathcal
O_\varepsilon(g_1),
\end{equation*}
\begin{equation*}\label{eqEx1_62Eps}
u(-\omega_0,r)-(1+t)u(0,r)=0,\quad
u(\omega_0,r)-(1-t)u(0,r)=0,\quad 0<r<\varepsilon,
\end{equation*}
where $(\omega,r)$ are the polar coordinates of the point $y$.
Further, writing the Laplace operator in the polar coordinates,
setting $\tau=\ln r$ and formally applying the Fourier transform
$u(\omega,\tau)\mapsto\tilde u(\omega,\lambda)$, we obtain the
following model nonlocal problem with the parameter
$\lambda\in\mathbb C$:
\begin{equation}\label{eqEx1_61Lambda}
\varphi''-\lambda^2\varphi=0,\quad |\varphi|<\omega_0,
\end{equation}
\begin{equation}\label{eqEx1_62Lambda}
\varphi(-\omega_0)-(1+t)\varphi(0)=0,\quad
\varphi(\omega_0)-(1-t)\varphi(0)=0,
\end{equation}
where $\varphi(\omega)=\tilde u(\omega,\lambda)$ for a fixed
$\lambda$. Clearly, the same problem corresponds to the point
$g_2$. The eigenvectors of problem~\eqref{eqEx1_61Lambda},
\eqref{eqEx1_62Lambda} are nonzero functions infinitely
differentiable on the segment $[-\omega_0,\omega_0]$ such that
they satisfy Eq.~\eqref{eqEx1_61Lambda} and nonlocal
conditions~\eqref{eqEx1_62Lambda}.

Straightforward calculation shows that the eigenvalues of
problem~\eqref{eqEx1_61Lambda}, \eqref{eqEx1_62Lambda} do not
depend on $t$ and have the form
\begin{equation}\label{eqEx1_Eig}
\lambda_k=\dfrac{\pi k}{\omega_0}i,\quad k=0,\pm1,\pm2,\dots\,.
\end{equation}
In the sequel, we consider the eigenvalues in the strip
$-1\le\Im\lambda\le0$. Since $0<\omega_0<\pi$, it follows that
this strip contains the unique eigenvalue $\lambda_0=0$. The
corresponding eigenvector has the form
\begin{equation}\label{eqEx1_EigVect}
\varphi_0(\omega)=-\frac{t}{\omega_0}\omega+1
\end{equation}
(up to a factor). It is easy to see that there is the unique (up
to the eigenvector) associate vector\footnote{If
$\lambda_0\in\mathbb C$ is an eigenvalue of
problem~\eqref{eqEx1_61Lambda}, \eqref{eqEx1_62Lambda} and
$\varphi_0(\omega)$ is the corresponding eigenvector, then the
associate vector $\varphi_1(\omega)$ is a solution (possibly,
zero) of the equation $\varphi_1''-\lambda_0^2\varphi_1+
\dfrac{d}{d\lambda}(\varphi_0''-\lambda^2\varphi_0)\Big|_{\lambda=\lambda_0}=0$
with nonlocal conditions~\eqref{eqEx1_62Lambda}. Therefore, if
$\lambda_0=0$, then the associate vector $\varphi_1(\omega)$ is a
solution of the equation $\varphi_1''=0$ with nonlocal
conditions~\eqref{eqEx1_62Lambda}.} $ \varphi_1(\omega)=0 $.

\begin{lemma}\label{lEx1_CodimUnstab}
There is a number $t_0>0$ such that $\codim\mathcal R(\mathbf
P_0)\le\codim\mathcal R(\mathbf P_t)=\const$ for $0<|t|\le t_0$.
\end{lemma}
\begin{proof}
1. Consider the operator
$$
\mathbf N_t:H_0^2(G)\to H_0^0(G)\times H_0^{3/2}(\Gamma_1)\times
H_0^{3/2}(\Gamma_2)
$$
given by
$$
\mathbf N_t=(\Delta u,\ u|_{\Gamma_1}-(1+t)
u(\Omega_1(y))|_{\Gamma_1},\ u|_{\Gamma_2}-(1-t)
u(\Omega_2(y))|_{\Gamma_2}).
$$
Since the line $\Im\lambda=-1$ contains no eigenvalues of
problem~\eqref{eqEx1_61Lambda}, \eqref{eqEx1_62Lambda}, it follows
from Theorem~3.4 in~\cite{SkMs86} that the operator $\mathbf N_t$
has the Fredholm property for all $t$. Further, the operator
$u\mapsto u(\Omega_j(y))|_{\Gamma_j}$ is a bounded operator acting
from $H_0^2(G)$ to $H_0^{3/2}(\Gamma_j)$. On the other hand, small
perturbations do not change the index of a Fredholm operator (see
Theorem~\ref{thAppBTheorSec16Kr}); therefore, $\ind\mathbf
N_t=\const$ for all $t$ from a sufficiently small neighborhood of
an arbitrary point $t'\in\mathbb C$, which yields
\begin{equation}\label{eqEx1_CodimUnstab_1}
\ind\mathbf N_t=\const,\quad t\in\mathbb C.
\end{equation}

2. Let us prove that
\begin{equation}\label{eqEx1_CodimUnstab_2}
\codim\mathcal R(\mathbf N_t)=\const,\quad |t|\le t_0,
\end{equation}
where $t_0>0$ is sufficiently small. Due
to~\eqref{eqEx1_CodimUnstab_1}, it suffices to show that
\begin{equation}\label{eqEx1_CodimUnstab_3}
\dim\ker\mathbf N_t=0,\quad |t|\le t_0.
\end{equation}

Let $t=0$, and let $u\in\ker \mathbf N_0$. Lemma~2.1
in~\cite{GurMatZam05} implies that the function $u$ is infinitely
differentiable outside an arbitrarily small neighborhood of the
set $\{g_1,g_2\}$. On the other hand, $u\in H_0^2(G)\subset
W_2^2(G)$; therefore, by the Sobolev embedding theorem, $u\in
C^\infty(G)\cap C(\overline{G})$ and
\begin{equation}\label{eqEx1_CodimUnstab_2'}
u(g_1)=u(g_2)=0.
\end{equation}

Since the coefficients of the problem are real for $t=0$, we can
assume without loss of generality that the function $u(y)$ is
real-valued. If the function $|u(y)|$ achieves its maximum inside
the domain $G$, then the maximum principle implies that $u=\const$
in $\overline G$; hence, $u=0$ by~\eqref{eqEx1_CodimUnstab_2'}. If
$|u(y)|$ achieves its maximum on the part $\Gamma_i$ of the
boundary, then nonlocal conditions~\eqref{eqEx1_62}, which take
the form
$$
u|_{\Gamma_1}=u(\Omega_1(y))|_{\Gamma_1},\qquad
u|_{\Gamma_2}=u(\Omega_2(y))|_{\Gamma_2}
$$
for $t=0$, imply that $|u(y)|$ also achieves its maximum inside
the domain $G$; hence, $u=0$ by what has been proved. Finally, if
$|u(y)|$ achieves its maximum at the point $g_1$ or $g_2$, then
$u=0$ by~\eqref{eqEx1_CodimUnstab_2'}.

Thus, we have proved that $\dim\ker\mathbf N_0=0$. It follows from
Theorem~\ref{thAppBTheorSec16Kr} that $\dim\ker\mathbf
N_t\le\dim\ker\mathbf N_0=0$ for sufficiently small $|t|$; this
yields~\eqref{eqEx1_CodimUnstab_3} and,
hence,~\eqref{eqEx1_CodimUnstab_2}.

3. Now we prove that
\begin{equation}\label{eqEx1_CodimUnstab_4}
\mathcal R(\mathbf P_t)=\{f\in L_2(G):\ (f,0,0)\in\mathcal
R(\mathbf N_t)\},\qquad 0\ne t\in\mathbb C.
\end{equation}

Since any solution $u\in H_0^2(G)$ of problem~\eqref{eqEx1_61},
\eqref{eqEx1_62} with right-hand side $f\in L_2(G)$ belongs to
$W_2^1(G)$, it follows that
\begin{equation}\label{eqEx1_CodimUnstab_5}
\mathcal R(\mathbf P_t)\supset\{f\in L_2(G):\ (f,0,0)\in\mathcal
R(\mathbf N_t)\},\qquad t\in\mathbb C.
\end{equation}

To prove the inverse embedding for $t\ne0$, we consider an
arbitrary function $f\in\mathcal R(\mathbf P_t)$. Let $u\in
W_2^1(G)$ be a solution of problem~\eqref{eqEx1_61},
\eqref{eqEx1_62} with the right-hand side $f$. It follows
from~\cite{GurTrMIRAN} that $u\in H_{a+1}^2(G)$ for all $a>0$. Due
to~\eqref{eqEx1_Eig}, there is a number $a>0$ such that the strip
$-1\le\Im\lambda<a$ contains the unique eigenvalue $\lambda_0=0$
of problem~\eqref{eqEx1_61}, \eqref{eqEx1_62}. Theorem~3.3
in~\cite{SkMs86} about the asymptotic behavior of solutions of
nonlocal problems implies that
\begin{equation}\label{eqEx1_CodimUnstab_5'}
u(y)=c_j\varphi_0(\omega)+d_j\varphi_0(\omega)\ln r +v_j(y),\quad
y\in G\cap\mathcal O_{\varepsilon}(g_j),
\end{equation}
where $(\omega,r)$ are the polar coordinates with the pole at the
point $g_j$, $\varphi_0(\omega)$ is given
by~\eqref{eqEx1_EigVect}, and $v_j\in H_0^2(G\cap\mathcal
O_{\varepsilon}(g_j))$. Note that
$$
u\in  W_2^1(G\cap\mathcal O_{\varepsilon}(g_j)),\qquad v_j\in
W_2^1(G\cap\mathcal O_{\varepsilon}(g_j)),
$$
but
$$
\varphi_0(\omega)\notin W_2^1(G\cap\mathcal
O_{\varepsilon}(g_j))\qquad \varphi_0(\omega)\ln r\notin
W_2^1(G\cap\mathcal O_{\varepsilon}(g_j))
$$
for $t\ne0$. Therefore, $c_j=d_j=0$
in~\eqref{eqEx1_CodimUnstab_5'}, which yields $u\in H_0^2(G)$.
Thus, we have proved that $(f,0,0)\in\mathcal R(\mathbf N_t)$,
i.e.,
\begin{equation}\label{eqEx1_CodimUnstab_5''}
\mathcal R(\mathbf P_t)\subset\{f\in L_2(G):\ (f,0,0)\in\mathcal
R(\mathbf N_t)\},\qquad 0\ne t\in\mathbb C.
\end{equation}

Relations~\eqref{eqEx1_CodimUnstab_5}
and~\eqref{eqEx1_CodimUnstab_5''}
imply~\eqref{eqEx1_CodimUnstab_4}.

4. Let us prove that\footnote{In~\eqref{eqEx1_CodimUnstab_6}, the
codimension of the subspace $\{f\in L_2(G):\ (f,0,0)\in\mathcal
R(\mathbf N_t)\}$ is calculated in the space $H_0^0(G)$, whereas
the codimension of the subspace $\mathcal R(\mathbf N_t)$ is
calculated in the space $H_0^0(G)\times H_0^{3/2}(\Gamma_1)\times
H_0^{3/2}(\Gamma_2)$.}
\begin{equation}\label{eqEx1_CodimUnstab_6}
\codim\{f\in L_2(G):\ (f,0,0)\in\mathcal R(\mathbf
N_t)\}=\codim\mathcal R(\mathbf N_t),\qquad t\in\mathbb C.
\end{equation}
Fix a number $t\in\mathbb C$ and set
$$
J_1=\codim\{f\in L_2(G):\ (f,0,0)\in\mathcal R(\mathbf
N_t)\},\qquad J_2=\codim\mathcal R(\mathbf N_t).
$$
Denote $\mathcal H_0^0(G,\Gamma)=H_0^0(G)\times
H_0^{3/2}(\Gamma_1)\times H_0^{3/2}(\Gamma_2)$.

Let $f\in L_2(G)$ and $(f,0,0)\in\mathcal R(\mathbf N_t)$. This is
equivalent to the relations
$$
\big((f,0,0),F_j\big)_{\mathcal H_0^0(G,\Gamma)}=0,\quad
j=1,\dots,J_2,
$$
where $F_j$ are functions that form the basis for the orthogonal
complement to the subspace $\mathcal R(\mathbf N_t)$ in the space
$\mathcal H_0^0(G,\Gamma)$. By the Riesz theorem on the general
form of a linear bounded functional in a Hilbert space, these
relations are equivalent to the following ones:
$$
(f,\hat f_j)_{L_2(G)}=0,\quad j=1,\dots,J_2,
$$
where $\hat f_j$ are some functions from the space $L_2(G)$. Thus,
\begin{equation}\label{eqLess}
J_1\le J_2
\end{equation}
(the equality takes place if and only if the functions $\hat f_j$
are linearly independent).

Let us prove the inverse inequality. Let $F=(f,f_1,f_2)$ be an
arbitrary function from the space $\mathcal R(\mathbf N_t)$. Then
there exists a function $u\in H_0^2(G)$ such that
\begin{equation*}
 \Delta u=f(y),\quad y\in G,
\end{equation*}
\begin{equation*}
u|_{\Gamma_1}-(1+t) u(\Omega_1(y))|_{\Gamma_1}=f_1,\qquad
u|_{\Gamma_2}-(1-t) u(\Omega_2(y))|_{\Gamma_2}=f_2.
\end{equation*}
Using Lemma~3.1 in~\cite{MP}, we can construct a function $v\in
H_0^2(G)$ such that
\begin{equation*}
v|_{\Gamma_1}-(1+t) v(\Omega_1(y))|_{\Gamma_1}=f_1,\qquad
v|_{\Gamma_2}-(1-t) v(\Omega_2(y))|_{\Gamma_2}=f_2,
\end{equation*}
\begin{equation}\label{eqNormvf}
\|v\|_{H_0^2(G)}\le
k_1\big(\|f_1\|_{H_0^{3/2}(\Gamma_1)}+\|f_2\|_{H_0^{3/2}(\Gamma_2)}\big),
\end{equation}
where $k_1>0$ does not depend on $f_1$ and $f_2$.

Clearly, the function $w=u-v\in H_0^2(G)$ is a solution of the
problem
\begin{equation*}
 \Delta w=f(y)-\Delta v,\quad y\in G,
\end{equation*}
\begin{equation*}
w|_{\Gamma_1}-(1+t) w(\Omega_1(y))|_{\Gamma_1}=0,\qquad
w|_{\Gamma_2}-(1-t) w(\Omega_2(y))|_{\Gamma_2}=0.
\end{equation*}
Therefore,
$$
f-\Delta v\in L_2(G),\qquad (f-\Delta v,0,0)\in\mathcal R(\mathbf
N_t),
$$
which is equivalent to the relations
$$
(f-\Delta v, f_j')_{L_2(G)}=0,\quad j=1,\dots,J_1,
$$
where $f_j'\in L_2(G)$ are functions that form the basis for the
orthogonal supplement to the subspace $\{f\in L_2(G):\
(f,0,0)\in\mathcal R(\mathbf N_t)\}$ in the space $L_2(G)$. By the
Riesz theorem on the general form of a linear bounded functional
in a Hilbert space and by estimate~\eqref{eqNormvf}, these
relations are equivalent to the following ones:
$$
(F,F_j')_{\mathcal H_0^0(G,\Gamma)}=0,\quad j=1,\dots,J_1,
$$
where $F_j'$ are some functions from the space $\mathcal
H_0^0(G,\Gamma)$. Thus,
\begin{equation}\label{eqGreater}
J_2\le J_1
\end{equation}
(the equality takes place if and only if the functions $F_j'$ are
linearly independent).

Inequalities~\eqref{eqLess} and \eqref{eqGreater}
imply~\eqref{eqEx1_CodimUnstab_6}.

It follows from relations~\eqref{eqEx1_CodimUnstab_6} and
\eqref{eqEx1_CodimUnstab_2} that
\begin{equation}\label{eqEx1_CodimUnstab_6*}
\codim\{f\in L_2(G):\ (f,0,0)\in\mathcal R(\mathbf
N_t)\}=\const,\qquad |t|\le t_0.
\end{equation}

Combining~\eqref{eqEx1_CodimUnstab_4}, \eqref{eqEx1_CodimUnstab_5}
for $t=0$, and~\eqref{eqEx1_CodimUnstab_6*}, we complete the
proof.
\end{proof}

\begin{lemma}\label{lEx1_kerUnstab}
Let the number $t_0>0$ be the same as in
Lemma~{\rm\ref{lEx1_CodimUnstab}}. Then $\dim\ker\mathbf
P_0>\dim\ker\mathbf P_t=0$ for $0<|t|\le t_0$.
\end{lemma}
\begin{proof}
1. Let $0<|t|\le t_0$, and let $u\in\ker\mathbf P_t$. Similarly to
item~3 of the proof of Lemma~\ref{lEx1_CodimUnstab}, one can show
that $u\in H_0^2(G)$; therefore, $u\in\ker \mathbf N_t$. It
follows from~\eqref{eqEx1_CodimUnstab_3} that $u=0$; thus,
$\dim\ker\mathbf P_t=0$.

2. Let $t=0$. In this case, $u=\const$ belongs to $\ker\mathbf
P_0$.
\end{proof}

\begin{proof}[Proof of Theorem~{\rm \ref{thEx1_IndUnstab}}]
Applying Lemmas~\ref{lEx1_kerUnstab} and \ref{lEx1_CodimUnstab},
we obtain
\begin{multline*}
 \ind\mathbf P_0=\dim\ker\mathbf
P_0-\codim\mathcal R(\mathbf P_0)>\\
>-\codim\mathcal R(\mathbf
P_0)\ge -\codim\mathcal R(\mathbf P_t) =\ind\mathbf P_t,\quad
0<|t|\le t_0.
\end{multline*}
\end{proof}

{\bf 3.} Now we show that the index of the unbounded operator may
change even if nonlocal terms are supported in an arbitrarily
small neighborhood of the conjugation points $g_1$ and $g_2$.

Let $G$, $\Gamma_i$, and $g_j$ be the same as above. Consider the
following nonlocal problem in the domain $G$:
\begin{equation}\label{eqEx2_61}
 \Delta u=f(y),\quad y\in G,
\end{equation}
\begin{equation}\label{eqEx2_62}
u|_{\Gamma_1}-(1+t) \xi(y) u(\Omega_1(y))|_{\Gamma_1}=0,\qquad
u|_{\Gamma_2}-(1-t) \xi(y) u(\Omega_2(y))|_{\Gamma_2}=0,
\end{equation}
where $\xi\in C^\infty(\mathbb R^2)$, the function $\xi$ is
supported in an arbitrarily small neighborhood of the points $g_1$
and $g_2$, and $\xi(y)=1$ near these points (see
Fig.~\ref{figEx2}).
\begin{figure}[ht]
{ \hfill\epsfxsize10cm\epsfbox{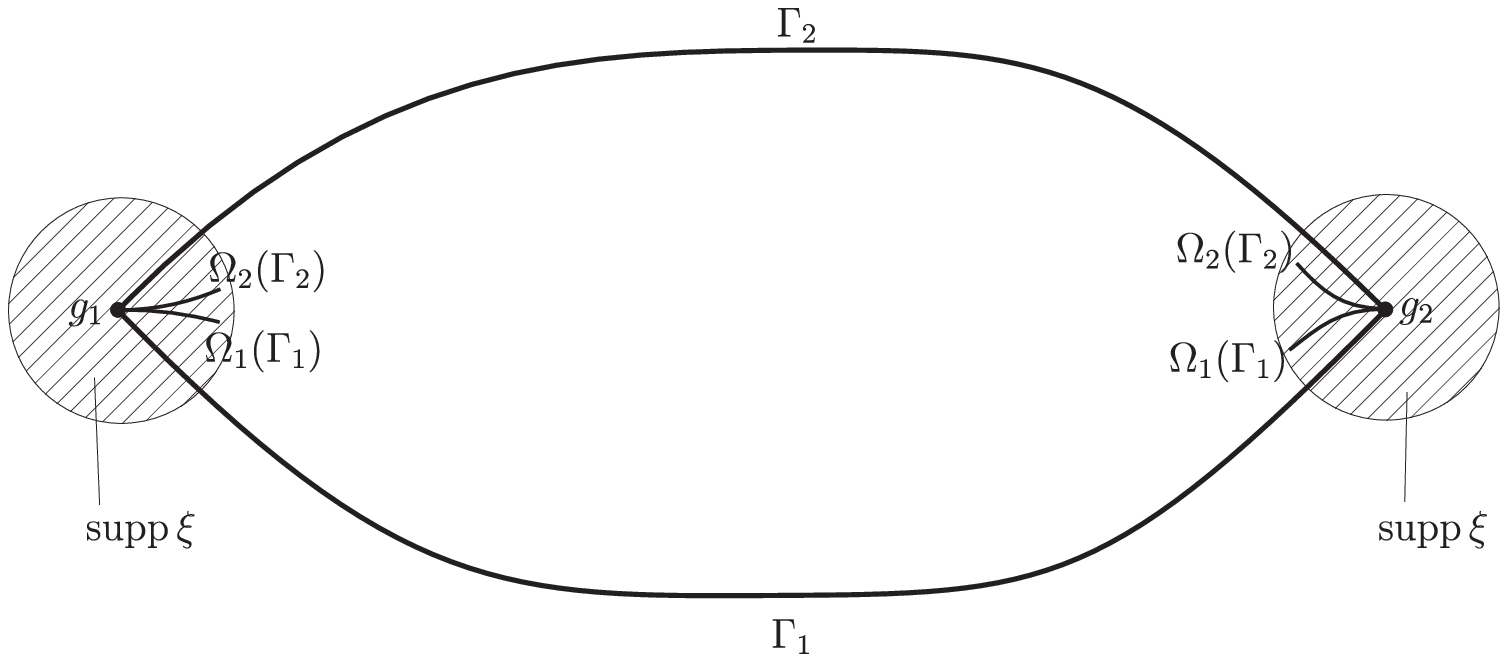}\hfill\ }
\caption{Problem~\eqref{eqEx2_61}, \eqref{eqEx2_62}}
\label{figEx2}
\end{figure}

Consider the unbounded operator $\mathbf P_t': \Dom(\mathbf
P_t')\subset L_2(G)\to L_2(G)$ given by
$$
 \mathbf P_t' u=\Delta u,\qquad u\in \Dom(\mathbf P_t),
$$
where
$$
\Dom(\mathbf P_t)=\{u\in
 W_2^1(G):\ \Delta u\in L_2(G)\ \text{and}\ u\ \text{satisfies~\eqref{eqEx2_62}}
 \}.
$$
By Theorem~2.1 in~\cite{GurMatZam05}, the operator $\mathbf P_t$
has the Fredholm property for all $t\in\mathbb C$.

The main result of this section is as follows.
\begin{theorem}\label{thEx2_IndUnstab}
There is a number $t_0>0$ such that $\ind \mathbf P_0'>\ind\mathbf
P_t'=\const$ for $0<|t|\le t_0$.
\end{theorem}
\begin{proof}
Nonlocal conditions~\eqref{eqEx1_62} differ from nonlocal
conditions~\eqref{eqEx2_62} by the operators
$$
u\mapsto(1+t) (1-\xi(y)) u(\Omega_1(y))|_{\Gamma_1},\qquad
u\mapsto(1-t)(1-\xi(y)) u(\Omega_2(y))|_{\Gamma_2}.
$$
Since the coefficients $(1\pm t) (1-\xi(y))$ at the nonlocal terms
vanish near the points $g_1$ and $g_2$, it follows that $
\ind\mathbf P_t'=\ind\mathbf P_t$ for all $t\in\mathbb C$ due
to~\cite{GurTrMIRAN}. Hence, the assertion of the theorem follows
from Theorem~\ref{thEx1_IndUnstab}.
\end{proof}

\section{The index instability in the case
where the nonlocal terms are supported outside the conjugation
points}

In this section, we show that the index of the unbounded operator
may also change in the case where the support of nonlocal terms
does not contain the conjugation points (and even lies strictly
inside the domain).

Let $G$, $\Gamma_i$, and $g_j$ be the same as above. We
additionally assume that
$$
 0<\omega_0<\pi/2
$$
and consider the following nonlocal problem in the domain $G$:
\begin{equation}\label{eqEx3_61}
 \Delta u=f(y),\quad y\in G,
\end{equation}
\begin{equation}\label{eqEx3_62}
u|_{\Gamma_1}+t u(\Omega(y))|_{\Gamma_1}=0,\qquad u|_{\Gamma_2}=0,
\end{equation}
where $t\in\mathbb R$ and $\Omega$ is a $C^\infty$-diffeomorphism
defined on some neighborhood of the curve $\Gamma_1$. Moreover,
let $ \overline{\Omega(\Gamma_1)}\subset G $ (see
Fig.~\ref{figEx3}).
\begin{figure}[ht]
{ \hfill\epsfxsize10cm\epsfbox{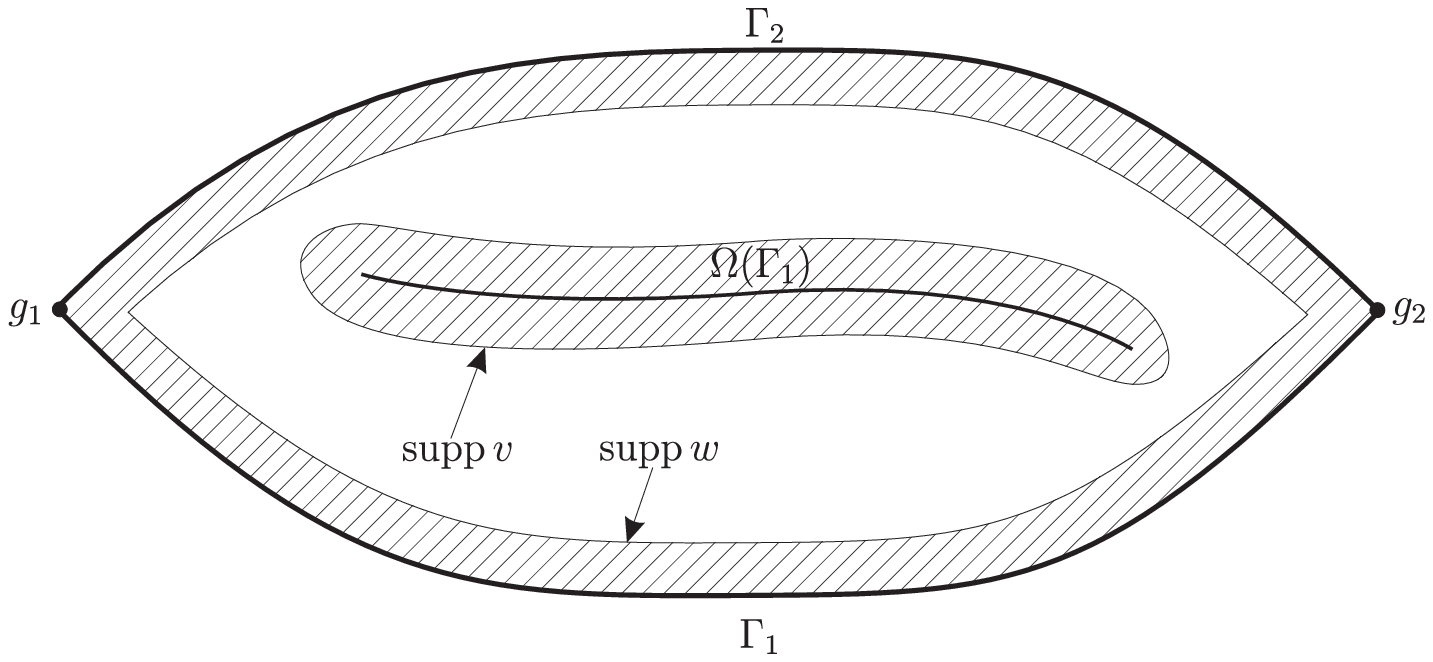}\hfill\ }
\caption{Problem~\eqref{eqEx3_61}, \eqref{eqEx3_62}}
\label{figEx3}
\end{figure}

Consider the unbounded operator $\mathbf P_t: \Dom(\mathbf
P_t)\subset L_2(G)\to L_2(G)$ given by
$$
 \mathbf P_t u=\Delta u,\qquad u\in \Dom(\mathbf P_t),
$$
where
$$
\Dom(\mathbf P_t)=\{u\in
 W_2^1(G):\ \Delta u\in L_2(G)\ \text{and}\ u\ \text{satisfies~\eqref{eqEx3_62}}
 \}.
$$
By Theorem~2.1 in~\cite{GurMatZam05}, the operator $\mathbf P_t$
has the Fredholm property for all $t\in\mathbb C$.

The main result of this section is as follows.
\begin{theorem}\label{thEx3_IndUnstab}
There is a number $t_0>0$ such that $0=\ind \mathbf
P_0>\ind\mathbf P_t$ for $0<|t|\le t_0$.
\end{theorem}

It is well known that the operator $\mathbf P_0$ is an
isomorphism; in particular,
\begin{equation}\label{eqEx3_IndP0=0}
\ind\mathbf P_0=0.
\end{equation}

Consider the operators $\mathbf P_t$. One and the same problem
\begin{equation}\label{eqEx1_61Lambda*}
\varphi''-\lambda^2\varphi=0,\quad |\varphi|<\omega_0,
\end{equation}
\begin{equation}\label{eqEx1_62Lambda**}
\varphi(-\omega_0)=\varphi(\omega_0)=0
\end{equation}
with the parameter $\lambda\in\mathbb C$ corresponds to the points
$g_1$ and $g_2$ (this problem is local because the nonlocal term
is supported outside the set $\{g_1,g_2\}$).

Straightforward calculation shows that the eigenvalues of
problem~\eqref{eqEx1_61Lambda*}, \eqref{eqEx1_62Lambda**} have the
form
\begin{equation}\label{eqEx3_Eig}
\lambda_k=\dfrac{\pi k}{2\omega_0}i,\quad k=\pm1,\pm2,\dots\,.
\end{equation}

\begin{lemma}\label{lEx3_kerPt=0}
We have $\dim\ker\mathbf P_t=0$ for $0<|t|\le 1$.
\end{lemma}
\begin{proof}
Let $u\in \ker\mathbf P_t$. Since $0<\omega_0<\pi/2$, it follows
from~\eqref{eqEx3_Eig} that the strip $-1\le\Im\lambda<0$ contains
no eigenvalues of problem~\eqref{eqEx1_61Lambda*},
\eqref{eqEx1_62Lambda**}. Therefore, $u\in W_2^2(G)$ by Theorem~1
in~\cite{GurDan04}. It follows from Lemma~2.1
in~\cite{GurMatZam05} that the function $u$ is infinitely
differentiable outside an arbitrarily small neighborhood of the
set $\{g_1,g_2\}$. Thus, the Sobolev embedding theorem implies
$u\in C^\infty(G)\cap C(\overline{G})$.

Since $t\in\mathbb R$, it follows that the coefficients in
problem~\eqref{eqEx3_61}, \eqref{eqEx3_62} are real; hence, we can
assume without loss of generality that the function $u(y)$ is
real-valued. If the function $|u(y)|$ achieves its maximum inside
the domain $G$, then $u=\const$ in $\overline G$ by the maximum
principle; in this case, $u=0$ by the second condition
in~\eqref{eqEx3_62}. If $|u(y)|$ achieves its maximum on
$\Gamma_1$, then the first condition in~\eqref{eqEx3_62} and the
relation $|t|\le1$ imply that $|u(y)|$ achieves its maximum inside
the domain $G$; in this case, $u=0$ by what has been proved.
Finally, if $|u(y)|$ achieves its maximum on
$\overline{\Gamma_2}$, then $u=0$ by its continuity and by the
second condition in~\eqref{eqEx3_62}.
\end{proof}

\begin{lemma}\label{lEx3_codimPt>0}
There is a number $t_0>0$ such that $\codim\mathcal R(\mathbf
P_t)>0$ for $0<|t|\le t_0$.
\end{lemma}
\begin{proof}
1. Consider the bounded operator
$$
\mathbf M_t: H_{a+1}^{2}(G)\to H_{a+1}^0(G)\times
H_{a+1}^{3/2}(\Gamma_1)\times H_{a+1}^{3/2}(\Gamma_2),\quad a>0,
$$
given by
$$
 \mathbf M_t=(\Delta u,\ u|_{\Gamma_1}+t
u(\Omega(y))|_{\Gamma_1},\ u|_{\Gamma_2}).
$$

Since the operator of embedding $W_2^{3/2}(\Gamma_1)\subset
H_{a+1}^{3/2}(\Gamma_1)$ is bounded by Lemma~\ref{lPsi=P1+Phi} and
$\overline{\Omega(\Gamma_1)}\subset G$, it follows that
\begin{equation}\label{eqEx3_kerPt=0_60}
\|u(\Omega(y))\|_{H_{a+1}^{3/2}(\Gamma_1)}\le
k_1\|u(y')\|_{W_2^{3/2}(\Omega(\Gamma_1))}\le
k_2\|u\|_{H_{a+1}^2(G)}.
\end{equation}
Therefore,
$$
\mathbf M_tu\in H_{a+1}^0(G)\times H_{a+1}^{3/2}(\Gamma_1)\times
H_{a+1}^{3/2}(\Gamma_2)
$$
whenever $u\in H_{a+1}^2(G)$ and $a>0$. Thus, the operator
$\mathbf M_t$ is well defined.

By Theorem~10.5 in~\cite{MP}, the local operator $\mathbf M_0$ is
an isomorphism for
\begin{equation}\label{eqEx3_kerPt=0_60'}
0<a<\pi/(2\omega_0).
\end{equation}

Fix a number $a$ satisfying~\eqref{eqEx3_kerPt=0_60'}. Since
$\mathbf M_0$ is an isomorphism and
estimate~\eqref{eqEx3_kerPt=0_60} is true, it follows that the
operator $\mathbf M_t$ is also an isomorphism for $0\le |t|\le
t_0$, provided $t_0=t_0(a)$ is sufficiently small.

2. Let us construct a function $u\in H_{a+1}^2(G)$ satisfying
nonlocal conditions~\eqref{eqEx3_62} and such that $$
u(\Omega(g_1))=1. $$ To this end, we consider a function $v\in
C^\infty(G)$ such that $v(y)=1$ for
$y\in\overline{\Omega(\Gamma_1)}$ and $\supp v\subset G$. In this
case, we have $v(\Omega(y))=1$ for $y\in\Gamma_1$; therefore, $
v(\Omega(y))|_{\Gamma_1}\in H_{a+1}^{3/2}(\Gamma_1). $

Further, we consider a function $w\in H_{a+1}^2(G)$ such that
$$
w|_{\Gamma_1}=-t v(\Omega(y))|_{\Gamma_1},\quad
w|_{\Gamma_2}=0,\qquad \supp
w\cap\overline{\Omega(\Gamma_1)}=\varnothing
$$
(the existence of such a function $w$ follows from Lemma~3.1
in~\cite{MP}). One can easily see that $u=v+w$ is the desired
function (see Fig.~\ref{figEx3}).

3. We approximate the function $f=\Delta u\in H_{a+1}^0(G)$ by the
functions $f_n\in L_2(G)$, $n=1,2,\dots$, in the space
$H_{a+1}^0(G)$:
\begin{equation}\label{eqEx3_kerPt=0_61}
\|f_n-f\|_{H_{a+1}^0(G)}\to0,\quad n\to\infty.
\end{equation}
If $\codim\mathcal R(\mathbf P_t)=0$ for $0<|t|\le t_0$, then, for
each function $f_n\in L_2(G)$, there exists a generalized solution
$u_n\in W_2^1(G)$ of problem~\eqref{eqEx3_61}, \eqref{eqEx3_62}
with the right-hand side $f_n$ (this solution is unique by
Lemma~\ref{lEx3_kerPt=0}). Moreover, $u_n\in H_{a+1}^2(G)$,
see~\cite{GurTrMIRAN}.

It follows from the fact that $\mathbf M_t$ is an isomorphism and
from~\eqref{eqEx3_kerPt=0_61} that
$$
\|u_n-u\|_{H_{a+1}^2(G)}\le k_3\|f_n-f\|_{H_{a+1}^0(G)}\to0,\quad
n\to\infty.
$$
Therefore, by the Sobolev embedding theorem, we have
\begin{equation}\label{eqEx3_kerPt=0_62}
u_n(\Omega(g_1))\to u(\Omega(g_1))=1,\quad n\to\infty.
\end{equation}
On the other hand, the strip $-1\le\Im\lambda<0$ contains no
eigenvalues of problem~\eqref{eqEx1_61Lambda*},
\eqref{eqEx1_62Lambda**}, and hence $u_n\in W_2^2(G)$ by Theorem~1
in~\cite{GurDan04}. By the Sobolev embedding theorem, $u_n\in
C(\overline{G})$, and it follows from the second nonlocal
condition in~\eqref{eqEx3_62} that $u_n(g_1)=0$. The first
nonlocal condition in~\eqref{eqEx3_62} now implies that $
u_n(\Omega(g_1))=0 $ (for $t\ne0$), which
contradicts~\eqref{eqEx3_kerPt=0_62}. Thus, we have proved that
$\codim\mathcal R(\mathbf P_t)>0$.
\end{proof}

Theorem~\ref{thEx3_IndUnstab} follows from~\eqref{eqEx3_IndP0=0}
and from Lemmas~\ref{lEx3_kerPt=0} and~\ref{lEx3_codimPt>0}.

\begin{remark} Let $\mathbf I$ denote the identity operator on $L_2(G)$, and
let $\lambda\in \mathbb C$. It is proved in~\cite{GurTrMIRAN} that
low-order terms in elliptic equation have no influence on the
index of the nonlocal operator $\mathbf P_t$. Therefore,
$$
\ind(\mathbf
P_t-\lambda\mathbf I)=\ind\mathbf P_t<0
$$
for $0<|t|\le t_0$, where $t_0>0$ is sufficiently small. Thus, the
spectrum of $\mathbf P_t$ for $0<|t|\le t_0$ coincides with the
whole complex plane.
\end{remark}

The author is grateful to A.~L.~Skubachevskii for attention.

\end{document}